\documentclass[11pt]{amsart}
\usepackage{amscd, amssymb}
\usepackage{graphics}

\theoremstyle{plain}

\newtheorem*{question}{Question}
\newtheorem*{conjecture}{Conjecture}

\numberwithin{equation}{section}

\newtheorem{proposition}[equation]{Proposition}
\newtheorem{theorem}{Theorem}
\Huge
\newtheorem{lemma}[equation]{Lemma}
\newtheorem{corollary}[equation]{Corollary}
\theoremstyle{definition}
\newtheorem{remark}[equation]{Remark}
\newtheorem{example}[equation]{Example}

\def    \R  {{\Bbb R}}
\def    \Z  {{\Bbb Z}}

\def    \CP {{\Bbb {CP}}}

\def    \C  {{\Bbb C}}
\def    \N  {{\Bbb N}}

\def    \index  {{\operatorname{index}}}
\def    \Tilde  {\widetilde}

\def    \ut     {\Tilde{u}}
\def    \Gt     {\Tilde{G}}

\begin{document}
\title[Hamiltonian circle actions with minimal isolated fixed points] {Hamiltonian circle actions with minimal isolated fixed points}

\author{Hui Li} 
\address{School of Mathematical Sciences\\
                Soochow University, Suzhou\\
                215006, China}
        \email{hui.li@suda.edu.cn}

\thanks{2010 classification.  53D05, 53D20, 55N25, 57R20, 32H02} 
\begin{abstract}
Let the circle act in a Hamiltonian fashion on a compact symplectic manifold $(M, \omega)$ of dimension $2n$. Then the $S^1$-action has at least $n+1$ fixed points.
We study the case when the fixed point set  consists of precisely $n+1$ isolated points. We first show certain equivalence on the first Chern class of $M$ and some particular weight of the $S^1$-action at some fixed point. 
Then we show that the particular weight  can completely determine the integral cohomology ring of $M$, the total Chern class of $M$, and the sets of weights of the $S^1$-action at all the fixed points. We will see that all these data are isomorphic to those of known examples, $\CP^n$, or $\Gt_2(\R^{n+2})$ with $n\geq 3$ odd,  equipped with standard circle actions. 
\end{abstract}

 \maketitle

\section{Introduction}
Let the circle act on a compact  $2n$-dimensional symplectic manifold $(M, \omega)$ with moment map $\phi$. The moment map $\phi$ is a perfect Morse-Bott function with critical set
being the fixed point set $M^{S^1}$ of the $S^1$-action.
Since the even Betti numbers of $M$, $b_{2i}(M)\geq 1$ for
all $0\leq 2i\leq 2n$, $M^{S^1}$ contains
at least $n+1$ points. When $M^{S^1}$ consists of exactly $n+1$ isolated points, we say that the action has {\bf minimal} isolated fixed points.
Two families of known examples of compact Hamiltonian $S^1$-manifolds with minimal isolated fixed points are $\CP^n$, and $\Gt_2(\R^{n+2})$ with $n \geq 3$ odd, see Examples~\ref{CP^n} and \ref{grasso}.  

When a manifold admits a group action, the geometry and topology of the manifold and the group action have constraints on each other. In this paper,
for Hamiltonian $S^1$-actions on compact symplectic manifolds with minimal isolated fixed points, we study the interplay between the geometry and topology of the symplectic manifolds and the $S^1$-actions.

For a symplectic $S^1$-manifold $(M, \omega)$ of dimension $2n$ with isolated fixed points, a neighborhood of each fixed point $P$ is $S^1$-equivariantly diffeomorphic to a neighborhood of the tangent space at $P$ with an $S^1$-linear action. The tangent space at $P$ splits into $n$ copies of $\C$, on each of which $S^1$ acts by multiplication by $\lambda^{w_i}$, where $\lambda\in S^1$ and $w_i\in\Z$ for
$i=1, \cdots, n$.
The integers $w_i$'s are well defined, and are called the {\it  weights} of the $S^1$-action at $P$. 

The topological and geometrical data we concern are the integral cohomology ring
and total Chern class of the manifold, and the data on the circle action we concern
are the sets of weights at the fixed points. The first Chern class of the manifold is an important data, especially for K\"ahler manifolds, in certain cases, it can determine the biholomorphism type of the manifold.

Let $(M, \omega)$ be a connected compact Hamiltonian $S^1$-manifold
of dimension $2n$ with minimal isolated fixed points, denoted $P_0,  P_1, \cdots, P_n$, let $\phi$ be the moment map. By Lemma~\ref{order}, the fixed points can be labelled so that
\begin{equation}\label{<}
\phi(P_0) < \phi(P_1) < \cdots < \phi(P_n).
\end{equation}
The moment map image is the closed interval $[\phi(P_0), \phi(P_n)]$.
By Lemma~\ref{k|}, if $[\omega]$ is an integral class, then 
$\phi(P_i)-\phi(P_j)$ is an integer for any $i$ and $j$.
In the statements of our main theorems below, we are tacitly referring to (\ref{<}), when we say that the integer  $\phi(P_n)-\phi(P_0)$ occurs as a weight, we mean that a weight of the $S^1$-action at some fixed point is equal to the integer $\phi(P_n)-\phi(P_0)$.

\begin{theorem}\label{equiv}
Let  $(M, \omega)$ be a compact $2n$-dimensional Hamiltonian $S^1$-manifold with fixed point set consisting of $n+1$ isolated points, $P_0,  P_1, \cdots, P_n$, and let $\phi\colon M\to\R$ be the moment map. Assume $[\omega]$ is a primitive integral class. Then for any $i\neq j$, $\phi(P_i)-\phi(P_j)$ is a nonzero integer, and
\begin{enumerate}
\item $c_1(M)=(n+1)[\omega]$ if and only if the integer $\phi(P_n)-\phi(P_0)$ occurs as a weight of the action at some fixed point,
\item $c_1(M)= n[\omega]$ if and only if  
$\phi(P_{n-1})-\phi(P_0)=\phi(P_n)-\phi(P_1)$ holds and this integer occurs as a weight of the action at some fixed point, but 
$\phi(P_n)-\phi(P_0)$ is not a weight at any fixed point. 
\end{enumerate}          
\end{theorem}

For the two cases as in Theorem~\ref{equiv}, we obtain Theorems~\ref{lw1} and \ref{lw2}. Theorem~\ref{lw1} follows from Propositions~\ref{0nring}, \ref{0ni}  and \ref{weights-tc1}. Theorem~\ref{lw2} follows from
Propositions~\ref{01ring}, \ref{01i} and \ref{weights-tc2}. 

\begin{theorem}\label{lw1}
Let  $(M, \omega)$ be a compact $2n$-dimensional Hamiltonian $S^1$-manifold with fixed point set consisting of $n+1$ isolated points, $P_0,  P_1, \cdots, P_n$, and let $\phi\colon M\to\R$ be the moment map. Assume $[\omega]$ is an integral class. If the integer $\phi(P_n)-\phi(P_0)$ occurs as a weight of the $S^1$-action at some fixed point, then all the following are true.
\begin{enumerate}
\item  The integral cohomology ring of $M$ is isomorphic to that of $\CP^n$.
\item The total Chern class of $M$ is isomorphic to that of $\CP^n$.
\item  The sets of weights of the $S^1$-action at all the fixed points on $M$ are isomorphic to those of a standard circle action on $\CP^n$ (as in Example~\ref{CP^n}).
\end{enumerate}
\end{theorem}

\begin{theorem}\label{lw2}
 Let  $(M, \omega)$ be a compact $2n$-dimensional Hamiltonian $S^1$-manifold with fixed point set consisting of $n+1$ isolated points, $P_0,  P_1, \cdots, P_n$, and let $\phi\colon M\to\R$ be the moment map. Assume $[\omega]$ is an integral class. If $\phi(P_{n-1})-\phi(P_0)=\phi(P_n)-\phi(P_1)$ holds and this integer occurs as a weight of the $S^1$-action at some fixed point, but the integer $\phi(P_n)-\phi(P_0)$ is not a weight at any fixed point, then $n\geq 3$ is odd, and all the following are true.
\begin{enumerate}
\item The integral cohomology ring of $M$ is isomorphic to that of  $\Gt_2(\R^{n+2})$.
\item The total Chern class of $M$ is isomorphic to that of  $\Gt_2(\R^{n+2})$.
\item  The sets of weights of the $S^1$-action at all the fixed points on $M$ are isomorphic to those of a standard circle action on $\Gt_2(\R^{n+2})$ (as in Example~\ref{grasso}).
\end{enumerate}
\end{theorem}

 If the symplectic manifold we consider is K\"ahler and the $S^1$-action is holomorphic,  by a theorem of Kobayashi and Ochiai \cite{KO}, and by our work in \cite[Prop. 4.2 and Sec. 5]{L}, we can conclude the following: under  the assumptions of Theorem~\ref{lw1}, $M$ is $S^1$-equivariantly biholomorphic to $\CP^n$, and is $S^1$-equivariantly  symplectomorphic to $\CP^n$; under the assumptions of Theorem~\ref{lw2}, $M$ is $S^1$-equivariantly biholomorphic to $\Gt_2(\R^{n+2})$, and is $S^1$-equivariantly symplectomorphic to 
$\Gt_2(\R^{n+2})$, with $n\geq 3$ odd.

\smallskip

In our theorems, we have three important data: the first Chern class $c_1(M)$,
the known weight, $\phi(P_n)-\phi(P_0)$, or $\phi(P_{n-1})-\phi(P_0)=\phi(P_n)-\phi(P_1)$ (but $\phi(P_n)-\phi(P_0)$ is not a weight), and the integral cohomology ring $H^*(M; \Z)$. (The total Chern class c(M) is a stronger condition than $c_1(M)$, and the set of all the weights is a stronger condition than the singer weight.) Having the equivalence in Theorem~\ref{equiv}, we can see that our theorems imply a few equivalent conditions. Moreover, one important new feature of this work is that we can use the known weight to determine the integral cohomology ring $H^*(M; \Z)$ without knowing all the weights of the $S^1$-action. Observe that one remaining interesting question is: if the integral cohomology ring $H^*(M; \Z)$ is isomorphic to that of $\CP^n$, then do we have any one of the other conditions appeared in Theorems~\ref{equiv} and \ref{lw1}? This is closely related to Petrie's conjecture (\cite{P}), which states that if a homotopy complex projective space $X$ admits a nontrivial circle action, then the total Pontryagin class of $X$ is isomorphic to that of $\CP^n$, i.e., $p(X)\cong p(\CP^n)$, where $2n=\dim X$.  If $(M, \omega)$ is a compact Hamiltonian $S^1$-manifold of dimension $2n$ with isolated fixed points, by \cite{L0}, $M$ is simply connected. If moreover $M$ has the integral cohomology ring of $\CP^n$, then $M$ is homotopy equivalent to $\CP^n$, and 
by Morse theory, this cohomology ring of $M$ implies that the $S^1$-action has precisely $n+1$ fixed points. Hence Petrie's conjecture, in one case, can be stated as follows:
\begin{conjecture}\label{conj}
Let $(M, \omega)$ be a compact symplectic manifold whose integral
cohomology ring is isomorphic to that of $\CP^n$. If $M$ admits a Hamiltonian $S^1$-action with isolated fixed points, then the total Chern class of $M$ is isomorphic to that of $\CP^n$, i.e., $c(M)\cong c(\CP^n)$.
\end{conjecture} 
In low dimensions, we can prove this conjecture, using the method 
of this paper. I hope that our method can help solving this conjecture, directly or by knowing simpler data $c_1(M)$ or the named weight.

We can also pose a corresponding question for the case when the integral cohomology ring of the manifold is isomorphic to that of $\Gt_2(\R^{n+2})$
with $n\geq 3$ odd. Here, instead, let us cite the more general question posed by Tolman in \cite{T} for Hamiltonian circle actions on symplectic manifolds:
\begin{question}
Consider a Hamiltonian circle action on a compact symplectic manifold
$(M, \omega)$ satisfying $H^{2i}(M; \R)=H^{2i}(\CP^n; \R)$ for all $i$. Is
$H^j(M; \Z)=H^{j}(\CP^n; \Z)$ for all $j$? Is the total Chern class $c(M)$ completely
determined by the integral cohomology ring $H^*(M; \Z)$?
\end{question}
 The condition $H^{2i}(M; \R)=H^{2i}(\CP^n; \R)$ for all $i$ in the question is also referred to as the manifold has minimal even Betti numbers, i.e., $b_{2i}(M)=1$ for all $0\leq 2i\leq 2n$. The question does not restrict to the case of isolated fixed points. If we consider the case that the fixed points are isolated, then by Morse theory, there are $n+1$ number of them,  and $H^j(M; \Z)=H^{j}(\CP^n; \Z)$ holds for all $j$, see Lemma~\ref{order}. 

Tolman answers the question above affirmatively in dimension 6, showing that there are $4$ possible kinds of integral cohomology ring structures and total Chern classes (\cite{T}); McDuff shows the existence of the two underlying manifolds other than $\CP^3$ and $\Gt_2(\R^5)$, which are called $V_5$ and $V_{22}$ (\cite{M}). In dimension 8, when there are 5 isolated fixed points, there is
a unique integral cohomology ring and total Chern class, those of $\CP^4$  (\cite{{GS}, {JT}}). The papers \cite{{LT}, {LOS}, {L}} study the question above 
for arbitrary dimensional manifolds when the fixed point set consists of $2$ and $3$ connected components, or study the diffeomorphism types of the manifolds.

Let $(M, \omega)$ be a compact  symplectic manifold of dimension $2n$. Suppose $H^2(M; \Z) = \Z$, and $[\omega]\in H^2(M; \Z)$ is a generator so that  $c_1(M)=k[\omega]$, with $k\in\Z_{\geq 0}$. Suppose $M$ admits a Hamiltonian circle action with $n+1$ isolated fixed points. Then $1\leq k\leq n+1$ (see Lemmas~\ref{gammaij} and \ref{mlarge}, also see \cite{H} and \cite[Prop. 7.5]{GS}). Hence, 
when the fixed points are isolated, suppose $\omega$ is scaled so that $[\omega]$ 
is primitive integral, this paper gives a positive answer to the question above  
in the following two cases: when $c_1(M)=(n+1)[\omega]$ or when $\phi(P_n)-\phi(P_0)$ is a weight, then there is a unique integral cohomology ring and total Chern class, those of $\CP^n$; when $c_1(M)=n[\omega]$ or when $\phi(P_{n-1})-\phi(P_0)= \phi(P_n)-\phi(P_1)$ is a weight and $\phi(P_n)-\phi(P_0)$ is not a weight, then there is a unique integral cohomology ring and total Chern class, those of $\Gt_2(\R^{n+2})$ with $n\geq 3$ odd.

The examples the author knows of the compact symplectic manifolds  which admit Hamiltonian circle actions with minimal isolated fixed points are $\CP^n$ with $n\geq 0$, $\Gt_2(\R^{n+2})$ with $n\geq 3$ odd, $V_5$ and $V_{22}$ in dimension $6$, and a coadjoint orbit of the Lie group $G_2$ in dimension $10$ (see \cite{Mo}). 

Now, let us look at another related work. In \cite{H}, for a compact {\it almost complex} $S^1$-manifold $M$ of dimension $2n$ with $n+1$ isolated fixed points, Hattori uses equivariant K-theory and equivariant cohomology to show that if $c_1(M)=(n+1)x$ 
modulo torsion or $c_1(M)=nx$ modulo torsion for some $x\in H^2(M; \Z)$, then the sets of weights at the fixed points are respectively isomorphic to
those of a standard circle action on $\CP^n$, or on  $\Gt_2(\R^{n+2})$ with $n\geq 3$ odd.  He also obtains rational cohomology information and the total Chern class of the manifold modulo torsion elements. 
In our case of Hamiltonian $S^1$-actions, with our Theorem~\ref{equiv} and Lemma~\ref{order}, using Hattori's results, we can derive our results in Theorems~\ref{lw1} and \ref{lw2}. However, for our case of Hamiltonian $S^1$-actions, starting from the one named weight, we give a direct  proof of Theorems~\ref{lw1} and \ref{lw2}  using symplectic tecniques. This symplectic proof is more geometric, it exhibits the direct and close relationship between the one named weight and all the other data mentioned in the theorems. Together with what we mentioned earlier, our method has its own advantages. I hope our approach can shed some light on the conjecture above.

The paper is organized as follows. In Section~\ref{prelim}, we present some preliminary materials for the paper. In Section~\ref{min}, we give examples of Hamiltonian $S^1$-manifolds with minimal isolated fixed points and prove a key lemma.  In Section~\ref{c1}, we prove Theorem~\ref{equiv}.  In Section~\ref{weightring},  we use the known weight to determine the integral cohomology ring of the manifold. In Section~\ref{weights}, we use the known weight to determine all the weights and to determine the total Chern class of the manifold.

\subsubsection*{Acknowledgement}  
 I would like to thank Sue Tolman for some helpful discussion. I thank the referee  for a suggestion which helps to improve the
exposition of the Introduction. This work is supported by the NSFC grant K110712116.

\section{Preliminaries}\label{prelim}

 In this section,  we briefly review equivariant cohomology, and we state and prove certain facts for compact Hamiltonian $S^1$-manifolds. 

 \smallskip

Let $M$ be a smooth
 $S^1$-manifold. The {\bf equivariant cohomology} of $M$ in a coefficient ring $R$ is
 $H^*_{S^1}(M; R) = H^*(S^{\infty}\times_{S^1} M; R)$, where
 $S^1$ acts on $S^{\infty}$ freely. If $p$ is a point, then $H^*_{S^1}(p; R)= H^*(\CP^{\infty}; R)=R[t]$, where $t\in H^2(\CP^{\infty}; R)$ is a generator.
 If $S^1$ acts on $M$ trivially, i.e., it fixes $M$, then $H^*_{S^1}(M; R)= H^*(M; R)\otimes R[t] =  H^*(M; R)[t]$. The projection map $\pi\colon S^{\infty}\times_{S^1} M\to \CP^{\infty}$ induces a pull back map
$\pi^*\colon H^*(\CP^{\infty}) \to H^*_{S^1}(M)$,
so $H^*_{S^1}(M)$ is an $H^*(\CP^{\infty})$-module.

 Let $(M, \omega)$ be a compact symplectic manifold. There exists an almost complex structure $J\colon TM\to TM$ which is
 {\bf compatible} with $\omega$, i.e., $\omega (J(\cdot ), \cdot)$ is a Riemannian metric. The space of compatible almost complex structures on $(M, \omega)$ is contractible, hence there is well defined total Chern class
$$c(M) = 1 + c_1(M) + \cdots + c_n(M)\in H^*(M; \Z),$$
where $c_i(M) \in H^{2i}(M; \Z)$ is the $i$-th Chern class of $M$.
Let $(M, \omega)$ be a compact symplectic $S^1$-manifold (the action preserves the symplectic form $\omega$), and $F$ be a connected component  of the fixed point set. The normal bundle to $F$ naturally splits into complex line bundles, one line bundle corresponding to one weight. If $(M, \omega)$ is a compact Hamiltonian $S^1$-manifold with moment map $\phi\colon M\to \R$, then $\phi$ is a perfect Morse-Bott function, its critical set coincides with the fixed point set of the action. At each connected component $F$ of the fixed point set, the negative normal bundle of $F$ is the sub-bundle with negative weights;  if $\lambda_F$ is the number of negative weights at $F$, counted with multiplicities,  then the {\bf Morse index of $F$}  is $\bf 2\lambda_F$, which is the dimension of the negative normal bundle to $F$.  Similarly, the {\bf Morse coindex of  $F$} is $2\lambda_F^+$, where 
$\lambda_F^+$ is the number of positive weights at $F$.

 Let $(M, \omega)$ be a compact $2n$-dimensional symplectic $S^1$-manifold. Assume the fixed points are isolated. Let $P$ be a fixed point, and let $\{w_1,  \cdots, w_n\}$ be the set of weights at $P$. We denote the equivariant total Chern class of $M$ as
 $$c^{S^1}(M) = 1 + c_1^{S^1}(M) + \cdots +  c_n^{S^1}(M)\in H^*_{S^1}(M; \Z),$$
where $c_i^{S^1}(M)\in  H^{2i}_{S^1}(M; \Z)$ is the $i$-th equivariant Chern class of $M$.
The restriction of $c^{S^1}(M)$ to $P$ is
$$c^{S^1}(M)|_P = 1 +  \sum_{i=1}^n c_i^{S^1}(M)|_P = 1  + \sum_{i=1}^n\sigma_i(w_1, \cdots,  w_n) t^i,$$
 where $\sigma_i(w_1, \cdots, w_n)$ is the $i$-th symmetric polynomial in the weights at $P$.

\

Next we state and prove some basic results.
First, the symplectic class $[\omega]$ of a Hamiltonian $S^1$-manifold can be extended to an equivariant cohomology class $\ut$: 
\begin{lemma}\cite[Lemma 2.3]{LT}\label{ut}
Let the circle act  on a compact symplectic manifold $(M,\omega)$
with moment map $\phi \colon M \to \R$.  Then there exists $\ut =[\omega -\phi t]\in H_{S^1}^2(M;\R)$ such that for any fixed point set component $F$,
$$\ut|_F =[\omega|_F]  - \phi(F)t.$$
If $[\omega]$ is an integral class, then $\ut$ is an integral class.  
\end{lemma}

For an $S^1$-manifold $M$,  when there exists a finite stabilizer group $\Z_k\subset S^1$, where $k > 1$, the set of points, $M^{\Z_k}$, which is pointwise fixed by $\Z_k$ but not pointwise fixed by $S^1$, is called a  {\bf $\Z_k$-isotropy submanifold}. If a $\Z_k$-isotropy submanifold is a sphere, it is called a  {\bf $\Z_k$-isotropy sphere}. For $k=1$, we have $M^{\Z_1}=M$.

\begin{lemma}\label{k|}
Let the circle act on a connected compact symplectic manifold
$(M, \omega)$ with moment map $\phi\colon M\to\R$. Assume  $[\omega]$ is an integral class. Then for any two fixed point set components $F$ and $F'$, $\phi(F) - \phi(F') \in \Z$.  If $\Z_k$ is the stabilizer group of some point on $M$,  then for any two fixed point set components $F$ and
$F'$ on the same connected component of $M^{\Z_k}$, we have $k\,|\left(\phi(F') - \phi(F)\right)$.
\end{lemma}

\begin{proof}
Since $M$ is compact, the action has at least two fixed components. Since $[\omega]$ is integral, by 
Lemma~\ref{ut}, we may assume $\phi(F)\in\Z$ for any fixed component $F$. Hence $\phi(F') - \phi(F)\in \Z$ for any two fixed components $F$ and $F'$.

 The isotropy submanifold $M^{\Z_k}$ is a compact Hamiltonian $S^1$-manifold, 
hence contains at least two fixed components. Consider the $S^1/\Z_k\approx S^1$ action on $M^{\Z_k}$, whose moment map is $\phi' = \phi/k$. Since $[\omega|_{M^{\Z_k}}]$ is integral,
 by the first paragraph, for any two fixed components $F$ and $F'$ on the same connected component of $M^{\Z_k}$, we have $\phi'(F') - \phi'(F) \in \Z$, i.e., $\frac{\phi(F')}{k} - \frac{\phi(F)}{k}\in \Z$.
\end{proof}

For a compact Hamiltonian $S^1$-manifold, Tolman gives an inequality on the index of a fixed point set component  \cite[Lemma 3.1]{T}. We only state it for the case of isolated fixed points:
\begin{lemma}\label{ind}
Let the circle act on a connected compact symplectic manifold
$(M, \omega)$ with isolated fixed points and with moment map $\phi\colon M\to\R$. Then for any fixed point $P$, $2\lambda_P\leq 2l$, where $l$ is the number of fixed points $Q$'s such that $\phi(Q) < \phi(P)$.
\end{lemma}

We will use the following lemma in the sequel.

\begin{lemma}\cite[Lemma 2.6]{T}\label{equalmod}
Let the circle act on a compact symplectic manifold $(M, \omega)$. Let $P$ and $Q$ be fixed points which lie on the same connected component of $M^{\Z_k}$ for some $k > 1$. Then the weights of the $S^1$-action at $P$ and $Q$ are equal modulo $k$.
\end{lemma}

\section{examples and basic results}\label{min}

In this section, we give standard examples of compact Hamiltonian $S^1$-manifolds with minimal isolated fixed points, we prove and set up key basic facts
for the next sections.

\begin{example}\label{CP^n}
Consider $\CP^n$. It naturally arises as a coadjoint orbit of $SU(n+1)$, hence it has a K\"ahler  structure and a Hamiltonian $SU(n+1)$ action.
Consider the $S^1\subset SU(n+1)$ action on $\CP^n$ given by
$$\lambda\cdot [z_0, z_1, \cdots, z_n] = [\lambda^{b_0}z_0, \lambda^{b_1}z_1, \cdots, \lambda^{b_n}z_n],$$
where the $b_i$'s are mutually {\it distinct} integers.
This action has $n+1$ isolated fixed points, $P_i = [0,\cdots, 0, z_i, 0, \cdots, 0]$, where $i=0, 1, \cdots, n$. The moment map values of the fixed points of the $S^1$-action are $\phi(P_i)=b_i$, $i=0, 1, \cdots, n$. The set of weights of the $S^1$-action at any $P_i$ is
$$\big\{w_{ij} = \phi(P_j) - \phi(P_i)\big\}_{j\neq i}.$$

As a ring, $H^*(\CP^n; \Z)=\Z[x]/x^{n+1}$, where $\deg (x) = 2$. The total Chern class $c(\CP^n)=(1+x)^{n+1}$, in particular, $c_1(\CP^n) = (n+1)x$.
\end{example}

\begin{example}\label{grasso}
Let $\Gt_2(\R^{n+2})$ be the  Grassmanian of oriented $2$-planes in $\R^{n+2}$, with $n\geq 3$ odd. It naturally arises as a coadjoint orbit of $SO(n+2)$, hence it has a K\"ahler  structure and a Hamiltonian $SO(n+2)$ action.
Consider the $S^1\subset SO(n+2)$ action on $\Gt_2(\R^{n+2})$ induced by the $S^1$-action on
$\R^{n+2}= \R\times\C^{\frac{n+1}{2}}$ given by
$$\lambda\cdot \big(t, z_0, \cdots, z_{\frac{n-1}{2}}\big) = \big(t, \lambda^{b_0}z_0, \cdots, \lambda^{b_{\frac{n-1}{2}}}z_{\frac{n-1}{2}}\big),$$
where the $b_i$'s, with $i\in \{0, \cdots, \frac{n-1}{2}\}$,  are mutually {\it distinct non-zero} integers. This action has $n+1$ isolated fixed points, denoted $P_0$,  $P_1$, $\cdots$, and $P_n$,  where for each $i\in \{0, \cdots, \frac{n-1}{2}\}$,  $P_i$ and $P_{n-i}$ are  given by the plane $(0, \cdots, 0, z_i, 0, \cdots, 0)$ respectively with two different orientations. The moment map values of the fixed points are respectively
$$\phi(P_0)=-b_0,  \cdots, \phi\big(P_{\frac{n-1}{2}}\big) = -b_{\frac{n-1}{2}},\, \phi\big(P_{\frac{n+1}{2}}\big) = b_{\frac{n-1}{2}},\, \cdots, \phi(P_n) = b_0,$$
assuming in the order of increasing.
The set of weights of the $S^1$-action at any
$P_i$, where $i\in\{0, 1, \cdots, n\}$,  is
$$\big\{w_{ij} = \phi(P_j)-\phi(P_i)\big\}_{j\neq i, n-i}\cup \big\{w_{i, n-i}=\frac{1}{2}\big(\phi(P_{n-i})-\phi(P_i)\big)\big\}.$$

As a ring, $H^*\big(\Gt_2(\R^{n+2}); \Z\big)=\Z[x, y]/(x^{\frac{n+1}{2}} - 2y, y^2)$,
where $\deg(x) =2$ and $\deg(y)= n+1$. The total Chern class $c\big(\Gt_2(\R^{n+2})\big) =\frac{(1+x)^{n+2}}{1+2x}$, in particular,
$c_1\big(\Gt_2(\R^{n+2})\big) = nx$.

(For the weights of this example, one may refer to \cite[Example 2.9]{T} for a $T^2$
action on $\Gt_2(\R^5)$ and derive the weights for the $S^1$-action; for higher dimensions, we can obtain the weights similarly. In Sec. 6 of \cite{H}, $\Gt_2(\R^{n+2})$, as a quadratic hypersurface in $\CP^{n+1}$, is given the standard circle action, Hattori writes down the weights of the $S^1$-action at the fixed points. From there, one can obtain the moment map values of the fixed points.)
\end{example}

Next, we prove a key basic lemma.
\begin{lemma}\label{order}
Let the circle act on a compact $2n$-dimensional symplectic manifold $(M, \omega)$ with moment map
 $\phi\colon M\to\R$. Assume $M^{S^1}$ consists of $n+1$ isolated points, denoted $M^{S^1}=\{P_0,  P_1, \cdots, P_n\}$. Then the points in $M^{S^1}$ can be labelled so that $P_i$  has Morse index $2i$, and we have $H^{2i}(M; \Z) = \Z$ and $H^{2i-1}(M; \Z)=0$ for all $0\leq 2i\leq 2n$. 
Moreover, 
\begin{equation}\label{eqorder}
\phi(P_0) < \phi(P_1) < \cdots < \phi(P_n). 
\end{equation}
\end{lemma}

\begin{proof}
Since $M$ is compact and symplectic, $0\neq [\omega^i]\in H^{2i}(M; \R)$, so
\begin{equation}\label{even}
\dim H^{2i}(M)\geq 1,\,\,\forall \,\,0\leq 2i\leq \dim (M).
\end{equation}
The moment map $\phi$ is a perfect Morse function. Since $\phi$ has $n+1$ critical points $P_0, P_1, \cdots, P_n$, each of which has even index, they must respectively have index $0$, $2$, $\cdots$, and $2n$ to make (\ref{even}) to hold.
Moreover, by Morse theory,  the negative disk bundle of $P_i$ is a $2i$-cell for the CW-structure of $M$ given by the Morse flow for any invariant metric, and there are no odd dimensional cells. By cellular cohomology theory,
$H^{2i}(M; \Z)\cong\Z$ and $H^{2i - 1}(M; \Z)=0$ for all $0\leq 2i\leq 2n$.     
The inequality (\ref{eqorder}) follows from Lemma~\ref{ind}.
\end{proof}

\begin{remark}
From now on, when we use the notation $M^{S^1}=\{P_0,  P_1, \cdots, P_n\}$, we mean that $P_i$ is the fixed point of
index $2i$, for any $0\leq i\leq n$.  We will use this fact and the inequality (\ref{eqorder}) so frequently as not to refer to Lemma~\ref{order}.
\end{remark}

For a compact Hamiltonian $S^1$-manifold $M$ with minimal isolated fixed points, by standard method in equivariant cohomology (see \cite{K} and \cite{TW}), 
we know that the extended classes of the equivariant Euler classes of the negative normal bundles of the fixed points form a basis of $H^*_{S^1}(M; \Z)$  as an $H^*(\CP^{\infty}; \Z)$-module. 
Moreover, since $H^*(M^{S^1}; \Z)$ has no torsion, we have $H^*(M; \Z) = H^*_{S^1}(M; \Z)/ (t)$ (see Sec.2 of \cite{LT}), hence the restriction of a basis of $H^*_{S^1}(M; \Z)$ to ordinary cohomology is a basis of $H^*(M; \Z)$. We summarize these as follows.

\begin{proposition}\label{equibase}
Let the circle act on a compact $2n$-dimensional symplectic manifold $(M, \omega)$ with moment map $\phi\colon M\to\R$. Assume $M^{S^1}=\{P_0,  P_1, \cdots, P_n\}$. Then
as an $H^*(\CP^{\infty}; \Z)$-module,  $H^*_{S^1}(M; \Z)$ has a basis 
$\big\{\Tilde\alpha_i\in H^{2i}_{S^1}(M; \Z)\,|\, 0\leq i\leq n\big\}$ such that
$$\Tilde\alpha_i|_{P_i} = \Lambda_i^- t^i, \,\,\mbox{and}\,\,\,\, \Tilde\alpha_i|_{P_j} =0,\,\,\forall\,\, j < i,$$
where $\Lambda_i^-$ is the product of the negative weights at $P_i$.
Moreover, $\big\{\alpha_i = r(\Tilde\alpha_i) \in H^{2i} (M; \Z)\,|\, 0\leq i\leq n\big\}$  is a basis for $H^*(M; \Z)$, where
$r\colon H^*_{S^1}(M; \Z)\to H^*(M; \Z)$ is the natural restriction map.
\end{proposition}

A direct corollary of Proposition~\ref{equibase} is:

\begin{corollary}\label{cor}
Let the circle act on a compact $2n$-dimensional symplectic manifold $(M, \omega)$ with moment map $\phi\colon M\to\R$. Assume $M^{S^1}=\{P_0,  P_1, \cdots, P_n\}$. Let 
$\Tilde\alpha\in H^{2i}_{S^1}(M; \Z)$ be a class such that $\Tilde\alpha|_{P_j}=0$ for all $j<i$. Then
$$\Tilde\alpha =  a_i \Tilde\alpha_i \,\,\,\mbox{for some $a_i\in\Z$}.$$
\end{corollary}

\section{On $c_1(M)$ and the proof of Theorem~\ref{equiv}}\label{c1}

In this section,  for compact Hamiltonian $S^1$-manifold $(M, \omega)$ with isolated fixed points, we first study the relation between $c_1(M)$ and the weights of the $S^1$-action at the fixed points, we then use the results to prove Theorem~\ref{equiv}.

In a symplectic $S^1$-manifold $(M, \omega)$ with isolated fixed points,
if $w>0$ is a weight of the $S^1$-action at a fixed point $P$, $-w$ is a weight of the $S^1$-action at a fixed point $Q$, and $P$ and $Q$ are on the same connected component of
$M^{\Z_w}$, we say that {\bf $w$ is a weight from $P$ to $Q$} or {\bf there is a weight $w$ from $P$ to $Q$}. 
When the signs of $w$ at $P$ and at $Q$ are clear, we will also say that {\bf there is a weight $\pm w$ between $P$ and $Q$}, or {\bf $\pm w$ is a weight between $P$ and $Q$}.

It is known (\cite{H}) that for an almost complex $S^1$-manifold with isolated fixed points, if $W^+$ and $W^-$ are respectively the set of positive weights and negative weights at all the fixed points, then 
$$W^- = - W^+.$$
So if $w$ is a weight at some fixed point, then $-w$ is also a weight at some fixed point. We will use this fact tacitly.

\begin{lemma}\label{sub}
Let the circle act on a connected compact symplectic manifold $(M, \omega)$ with moment map 
$\phi\colon M\to\R$. If $c_1(M)=k[\omega]$, then for any two fixed point set components
$F$ and $F'$,  we have $\Gamma_F - \Gamma_{F'} = k\big(\phi(F')-\phi(F) \big)$,
where $\Gamma_F$ and $\Gamma_{F'}$ are respectively the sums of the weights at $F$ and $F'$.
\end{lemma}
\begin{proof}
We have an equivariant extension $c_1^{S^1}(M) = k\ut +a\, t$  of $c_1(M)=k[\omega]$,
where $\ut$ is as in Lemma~\ref{ut} and $a \in\R$.  Let $f\in F$ and $f'\in F'$ be points. Then
$$c_1^{S^1}(M)|_f = \Gamma_F t = -k\phi(F)t + a\, t,\,\,\,\mbox{and}\,\,\,
c_1^{S^1}(M)|_{f'} = \Gamma_{F'} t= -k\phi(F')t + a\, t.$$
Subtracting the two equalities, we obtain the claim.
\end{proof}

\begin{lemma}\label{gammaij}
Let the circle act on a connected compact symplectic manifold $(M, \omega)$ with moment map $\phi\colon M\to\R$. Assume $M^{S^1}$ consists of isolated points. Let $P, Q\in M^{S^1}$ with $P\neq Q$, where $\index (P)=2i$ and $\index (Q) = 2j$ with $i \leq j$. Assume there is a weight $w$ from $P$ to $Q$,  $-w$ is not a weight at $P$, $-w$  has multiplicity $1$ at $Q$, and $w$ is the largest among the absolute values of all the weights at $P$ and $Q$. If $c_1(M) = k [\omega]$,  then $j-i +1 = k\frac{\phi(Q)-\phi(P)}{w}$.
\end{lemma}

\begin{proof}
Let 
$$W_P^- = \big\{a_1, \cdots, a_i\big\}\,\,\,\mbox{and}\,\,\,  W_P^+ = \big\{a_{i+1}, \cdots, a_n\big\}$$
 be respectively the set of negative weights and positive weights at $P$, and  
$$W_Q^- = \big\{b_1, \cdots, b_j\big\}\,\,\,\mbox{and}\,\,\, W_Q^+ = \big\{b_{j+1}, \cdots, b_n\big\}$$
 be respectively the set of negative weights and positive weights at $Q$.
Assume
$$a_n = w \in W_P^+, \,\,\,\mbox{and}\,\,\,  b_j = -w\in W_Q^-.$$
By Lemma~\ref{equalmod}, 
\begin{equation}\label{modw} 
W_P^-\cup W_P^+ = W_Q^- \cup W_Q^+  \mod w.
\end{equation}
Since $w$ is the largest among the absolute values of all the weights at $P$ and $Q$,  
$-w$ does not occur at $P$ and $-w$ occurs once at $Q$, 
up to a reordering of indices,  (\ref{modw}) can only yield the following relations:
$$a_1 = b_1, \,\,\,\cdots,\,\,\, a_m = b_m \,\,\,\mbox{for some $m$ with $0\leq m \leq i$},$$
$$a_{i+1} = b_{j+1},\,\,\, \cdots,\,\,\, a_{i+l} = b_{j+l} \,\,\,\mbox{for some $l$
with $0\leq l\leq n-j$},$$
$$a_{m+1}  = b_{j +l+1} - w, \,\,\,\cdots, \,\,\,  a_i = b_n-w,$$
$$a_{i+l+1}  = b_{m+1} + w,\,\,\, \cdots, \,\,\, a_{n-1}  = b_{j-1} + w, \,\,\,\mbox{and}$$
$$ a_n = b_j + 2 w.$$
 Let $\Gamma_P$ and $\Gamma_Q$ be respectively the sums of weights at $P$ and at $Q$. Then
 $$\Gamma_P - \Gamma_Q = \sum_{p=1}^n a_p - \sum_{p=1}^n b_p = (j-i +1)w.$$
Combining Lemma~\ref{sub}, we obtain the claim of the lemma.
\end{proof}

\begin{lemma}\label{mlarge}
Let the circle act on a compact symplectic manifold $(M, \omega)$ with moment map $\phi\colon M\to\R$. Assume $M^{S^1}$ consists of isolated points.  Let $w>0$ be the largest among all the weights at all the fixed points. Then there are $P, Q\in M^{S^1}$ with $\phi(P) < \phi(Q)$, such that  there is a weight $w$ from $P$ to $Q$,  $-w$ is not a weight at $P$ and $-w$ has multiplicity $1$ at $Q$.
\end{lemma}
\begin{proof}
Relative to the value of $\phi$, choose the lowest fixed point $P$ which has a weight $w$. Then $-w$ cannot be a weight at $P$. We argue this fact as follows. Suppose $-w$ is a weight at $P$, consider the connected component $C$ of $M^{\Z_w}$ containing $P$. Since $w$ is the largest weight on $M$, each weight at the fixed points in $C$ is $\pm w$. Since $P$ is of index at least $2$ in $C$, $C$ is compact symplectic, and $S^1$-Hamiltonian with moment map $\phi|_C$, 
$C$ must contain a minimum for $\phi|_C$, a fixed point $P'$ which has a lower moment map value and a positive weight $w$, contradicting to the choice of $P$.

Next, relative to the value of $\phi$, choose a closest fixed point $Q$ to $P$ such that $-w$ is a weight at $Q$,  and $P$ and $Q$ are on the same connected component $C$ of $M^{\Z_w}$. The moment map $\phi|_C$ on $C$  has a unique minimum (by the connectivity theorems).
Suppose $\phi(Q) \leq \phi(P)$, then $P$ is not the minimum for $\phi|_C$ on $C$, hence has a weight 
$-w$, contradicting to the first paragraph. Hence $\phi(P) < \phi(Q)$. 
Suppose $-w$ has multiplicity bigger than $1$ at $Q$, then in $C$, $P$ is the minimum, $Q$ has index at least $4$; since $C$ is symplectic, $C$ contains at least an index $2$ fixed point $Q'$ with $\phi|_C (P) < \phi|_C (Q') < \phi|_C (Q)$ (see Lemma~\ref{ind}), so $-w$ is a weight at $Q'$, contradicting to the choice of $Q$.
\end{proof}

Now we are ready to prove Theorem~\ref{equiv}.

\begin{proof}[Proof of Theorem~\ref{equiv}]
First, since $[\omega]$ is integral, by Lemma~\ref{k|}, $\phi(P_i)-\phi(P_j)\in\Z$ for any $i$ and $j$;  if $a$ is a weight between $P_i$ and $P_j$, then $a\,|\,\big(\phi(P_i)-\phi(P_j)\big)$.  By Lemma~\ref{order}, $\phi(P_i)-\phi(P_j)\neq 0$ for any $i\neq j$, and $H^2(M; \Z)=\Z$, together with 
the fact that $[\omega]$ is primitive integral, we know that there exists $k\in\Z$ such that $c_1(M)= k[\omega]$.

Let $w >0$ be the largest among all the weights at all the fixed points.  By Lemma~\ref{mlarge}, $w$ is a weight from some $P_l$ to some $P_m$ with $l < m$, and so $w\,|\,\big(\phi(P_l)-\phi(P_m)\big)$.

{\bf Case 1}.  If $c_1(M) = (n+1) [\omega]$, then by Lemma~\ref{gammaij}, $l=0$, $m=n$, and $w=\phi(P_n)-\phi(P_0)$. Conversely,
if the number $\phi(P_n)-\phi(P_0)$ is a weight at some fixed point, then by the first paragraph, it can only be a weight between $P_0$ and $P_n$, and it is equal to the largest weight $w$. Then Lemma~\ref{gammaij} implies $c_1(M) = (n+1) [\omega]$.

{\bf Case 2}.  Assume $c_1(M) = n [\omega]$. Then $n > 1$, since if $n=1$, $M$ can only be $S^2$ and $c_1(M)=2[\omega]$, a contradiction. So
$n\nmid (n+1)$, and by Lemma~\ref{gammaij}, we have $l=0$, $m=n-1$ and $w=\phi(P_{n-1})-\phi(P_0)$, or
$l=1$, $m=n$ and $w=\phi(P_n)-\phi(P_1)$, and both must hold by symmetry. Conversely, the hypotheses imply that
$w=\phi(P_{n-1})-\phi(P_0)=\phi(P_n)-\phi(P_1)$, and by the first paragraph, we have $l=0$ and $m=n-1$ or $l=1$ and $m=n$. Then  Lemma~\ref{gammaij} gives $c_1(M) = n [\omega]$.
\end{proof}

\section{from the known weight to the ring $H^*(M; \Z)$}\label{weightring}

In this section, we use the named weight to 
determine  the integral cohomology ring of the manifold, i.e., we prove
$(1)$ in Theorems~\ref{lw1} and \ref{lw2}. 

\subsection{The case when $\phi(P_n)-\phi(P_0)$ is a weight}

\
\medskip

In this part, we will first determine the sets of weights at $P_0$ and $P_n$,
then we use that to determine the integral cohomology ring of $M$. We first prove couple of basic lemmas.

Let us keep in mind the inequality (\ref{eqorder}) in Lemma~\ref{order}. We will use it without referring to it.

\begin{lemma}\label{mul1}
Let the circle act on a compact $2n$-dimensional symplectic manifold
$(M, \omega)$ with moment map $\phi\colon M\to\R$. Assume $[\omega]$ is an integral class and $M^{S^1} = \{P_0, P_1, \cdots, P_n\}$. For fixed indices $i$ and 
$j$, assume $a >0$  is a weight from $P_i$ only to some $P_l$ with $i < l < j$, $b >0$ is a weight from only some $P_m$ to $P_j$ with $i< m < j$, and
\begin{equation}\label{sum}
 a = - b + \big(\phi(P_j)-\phi(P_i)\big).
\end{equation}
Then (as numbers)
\begin{equation}\label{ww'}
a = \phi(P_k)-\phi(P_i)\,\,\,\mbox{and}\,\,\, -b = \phi(P_k)-\phi(P_j), \,\,\,\mbox{where $k\in\{l, m\}$}.
\end{equation}
Moreover, in any of the following $2$ cases, the multiplicity of such $a$ and $-b$ is 1.
\begin{enumerate}
\item  If  $a\geq\frac{1}{2}\big(\phi(P_j)-\phi(P_i)\big)$, then the submanifold $M^{\Z_{a}}$
has $P_i$ as the minimum. 
\item If $b\geq\frac{1}{2}\big(\phi(P_j)-\phi(P_i)\big)$, then the submanifold $M^{\Z_b}$ has $P_j$ as the maximum.
\end{enumerate}
\end{lemma}

\begin{proof}
Since $a$ is a weight from $P_i$ to $P_l$, and $b$ is a weight from $P_m$ to $P_j$, by Lemma~\ref{k|}, there exist $k_1, k_2\in\N$ such that
$$k_1 a = \phi(P_l)-\phi(P_i), \,\,\,\mbox{and}\,\,\, k_2 b = \phi(P_j)-\phi(P_m).$$
Taking the sum of these equalities and considering together with (\ref{sum}), it follows that $k_1=1$ or $k_2=1$ or both, in either case, (\ref{ww'}) follows.

Assume $a\geq \frac{1}{2}\big(\phi(P_j)-\phi(P_i)\big)$. In $M^{\Z_a}$, each weight is a multiple of $a$, $P_i$ is the minimal fixed point; since $a = \phi(P_k)-\phi(P_i)$, $P_k$ must be the next fixed point to $P_i$. Since by assumption $a$ is a weight from $P_i$ only to some $P_l$ with $i < l < j$, and $a\geq \frac{1}{2}\big(\phi(P_j)-\phi(P_i)\big)$, $a$ can only divide $\phi(P_k)-\phi(P_i)$, so
 $a$ can only be a weight from $P_i$ to $P_k$. If $a$ has multiplicity bigger than $1$, then it contradicts to  that $M^{\Z_a}$ is a symplectic manifold (or contradicts to Lemma~\ref{ind}).  The multiplicity of $b$ is the same as that of $a$. 
For the second case, the claim follows similarly by looking at $-\phi$.
\end{proof}

\begin{remark}
In Lemma~\ref{mul1}, in (\ref{ww'}),  we do not claim that $a$ is a weight between $P_i$ and $P_k$, and $b$ is a weight between $P_k$ and $P_j$. By assumption, we know at least one of them is true. When $l=m$, then both are true.
\end{remark}

\begin{lemma}\label{nother}
Let the circle act on a compact $2n$-dimensional symplectic manifold
$(M, \omega)$ with moment map $\phi\colon M\to\R$. Assume  $[\omega]$ is an integral class and 
$M^{S^1} = \{P_0, P_1, \cdots, P_n\}$. Assume there is a weight $a=\phi(P_i)-\phi(P_0)$ from $P_0$ to $P_i$ for some $i\neq 0$. Then for any other weight $a'$ at $P_0$, which  could be equal to $a$, there exists $j\neq 0, i$, such that $a'$ is a weight from $P_0$ to $P_j$. A similar claim holds if we replace $P_0$ by $P_n$.
\end{lemma}

\begin{proof}
Let $C$ be a connected component of $M^{\Z_{a'}}$ containing $P_0$. Using Lemma~\ref{JT} below by Jang and Tolman on the symplectic manifold $C$, 
we know that there exists an index $2$ fixed point  in $C$, say $P_j$, such that $a'$ is a weight from $P_0$ to $P_j$.  If $j=i$, then by Lemma~\ref{k|}, $a' | a$, so 
$\index |_C (P_j) \geq 4$,  a contradiction. We can similarly argue using $-\phi$ if we replace $P_0$ by $P_n$.
\end{proof}

\begin{lemma}\cite{JT}\label{JT}
Let the circle act on a closed $2n$-dimensional almost
complex manifold  with isolated fixed points. Let $w$ be the smallest
positive weight that occurs at the fixed points on $M$. Then given any 
$k \in\{0, 1, . . . , n-1\}$, the number of times the weight $-w$ occurs at
fixed points of index $2k+2$ is equal to the number of times the 
weight $+w$ occurs at fixed points of index $2k$.
\end{lemma}

\medskip
We obtain the sets of weights at $P_0$ and $P_n$ as follows.

\begin{lemma}\label{0n1}
Let the circle act on a compact $2n$-dimensional symplectic manifold
$(M, \omega)$ with moment map $\phi\colon M\to\R$. Assume $[\omega]$ is an integral class and $M^{S^1}=\{P_0,  P_1, \cdots, P_n\}$. If the integer $w_{0n}=\phi(P_n)-\phi(P_0)$ occurs as a weight at some fixed point,  then the set of weights at $P_0$ is
\begin{equation}\label{w0}
\big\{w_{0i} = \phi(P_i) -\phi(P_0)\big\}_{0 < i\leq n},
\end{equation}
and the set of weights at $P_n$ is
\begin{equation}\label{wn}
\big\{w_{ni} = \phi(P_i) -\phi(P_n)\big\}_{0\leq i < n}.
\end{equation}
\end{lemma}

\begin{proof}
Since $w_{0n}$ only divides $\phi(P_n)-\phi(P_0)$, it can only be a weight between $P_0$ and $P_n$. Moreover, it must have multiplicity $1$ (otherwise there would exist a symplectic manifold of dimension at least $4$ with only $2$ fixed points $P_0$ and $P_n$).

Let $\big\{\bar w_{0i}\,|\, 1\leq i\leq n\big\}$ be the set of weights at $P_0$, and $\big\{\bar w_{ni}\,|\, 0\leq i\leq n-1\big\}$ be the set of weights at $P_n$.  By Lemma~\ref{nother}, each $\bar w_{0i}\neq w_{0n}$ is a weight between $P_0$ and some $P_l$, with $0 < l < n$, and  each $\bar w_{ni}\neq w_{n0} = - w_{0n}$ is a weight between $P_n$ and some $P_m$, with $0 < m < n$. By Lemma~\ref{equalmod},
$$\big\{\bar w_{0i}\,|\, 1\leq i\leq n\big\} = \big\{\bar w_{ni}\,|\, 0\leq i\leq n-1\big\}\mod w_{0n}.$$
So for any $\bar w_{0i}\neq w_{0n}$, there exists a weight at $P_n$, (up to a change of index) let us call it $\bar w_{ni}\neq w_{n0}$, such that
$$\bar w_{0i} =\bar w_{ni} \mod w_{0n}.$$
Since $\bar w_{0i} < w_{0n}$ and $|\bar w_{ni}| < w_{0n}$, we have
\begin{equation}\label{e2}
\bar w_{0i} = \bar w_{ni}  +  w_{0n}.
\end{equation}
By Lemma~\ref{mul1},  there is $k$ with $0 < k < n$ such that
$\bar w_{0i} = \phi(P_k)-\phi(P_0)$ and $\bar w_{ni} = \phi(P_k)-\phi(P_n)$. We rename
$\bar w_{0i}$ as $w_{0k}$, and $\bar w_{ni}$ as $w_{nk}$, i.e., we denote
$$w_{0k} = \phi(P_k)-\phi(P_0)\,\,\,\mbox{and}\,\,\,  w_{nk} = \phi(P_k)-\phi(P_n).$$
Moreover, Lemma~\ref{mul1} also implies that $w_{0k}$ and $w_{nk}$ both have multiplicity $1$.
Since  there are $n$ number of weights at $P_0$, and any weight is of the form above for some index $k$
with multiplicity $1$, the set of weights at $P_0$ must be as claimed in (\ref{w0}).  Similarly, the set of weights at $P_n$ must be as claimed in (\ref{wn}).
\end{proof}

Next, we will use Lemma~\ref{0n1} to determine the ring $H^*(M; \Z)$.

\begin{lemma}\label{ringprod}
Let the circle act on a compact $2n$-dimensional symplectic manifold
$(M, \omega)$ with moment map $\phi\colon M\to\R$. Assume $[\omega]$ is an integral class and $M^{S^1}=\big\{P_0,  P_1, \cdots, P_n\big\}$. Then the following conditions are equivalent:
\begin{enumerate}
\item $H^*(M; \Z) = \Z[x]/x^{n+1}$, where $x=[\omega]$.
\item  $\Lambda_n^- =\prod_{j=0}^{n-1}\big(\phi(P_j)-\phi(P_n)\big)$, where
$\Lambda_n^-$ is the product of the (negative) weights at $P_n$.
\end{enumerate}
\end{lemma}

\begin{proof}
By Lemma~\ref{ut}, we have
 $\prod_{j=0}^{n-1}\big(\ut +\phi(P_j)t\big)\vert_{P_i} = 0$ for all  $i < n$.  
Then Corollary~\ref{cor} gives
\begin{equation}\label{atn}
 \prod_{j=0}^{n-1}\big(\ut +\phi(P_j)t\big) = a_n \Tilde\alpha_n,\,\,\mbox{with $a_n\in\Z$}.
\end{equation}
Restricting (\ref{atn})  to $P_n$ and using Proposition~\ref{equibase}, we get
$$a_n\Lambda_n^- =\prod_{j=0}^{n-1}\big(\phi(P_j)-\phi(P_n)\big).$$ 
Restricting (\ref{atn}) to ordinary cohomology, we get the generator of $H^{2n}(M; \Z)$:
$$\alpha_n =\frac{1}{a_n} [\omega]^n.$$
Hence $(2)$ holds if and only if $a_n =1$, and if and only if  $[\omega]^n$ is a generator of $H^{2n}(M; \Z)$. So $(1)$ implies $(2)$. Conversely, if we have $(2)$,
knowing that $[\omega]^n\in H^{2n}(M; \Z)$ is a generator, and $[\omega]$ is integral, we can see that each $[\omega]^i$ with $0\leq i\leq n$ is integral
and primitive. Together with the facts  $H^{2i}(M; \Z)=\Z$ and $H^{2i-1}(M; \Z)=0$ for all $0\leq i\leq n$ in Lemma~\ref{order}, we obtain $(1)$.
\end{proof}

\begin{remark}
The conditions in Lemma~\ref{ringprod} are also equivalent to $\Lambda_i^- =\prod_{j=0}^{i-1}\big(\phi(P_j)-\phi(P_i)\big)$  for $\forall \,\, i$. The proof is by considering
the classes $\prod_{j=0}^{i-1}\big(\ut +\phi(P_j)t\big)$ similarly as above.
\end{remark}

Using Lemmas~\ref{0n1} and \ref{ringprod}, we obtain the ring $H^*(M; \Z)$:
\begin{proposition}\label{0nring}
Let the circle act on a compact $2n$-dimensional symplectic manifold
$(M, \omega)$ with moment map $\phi\colon M\to\R$. Assume  $[\omega]$ is an integral class and $M^{S^1}=\big\{P_0,  P_1, \cdots, P_n\big\}$. Assume the integer $w_{0n}=\phi(P_n)-\phi(P_0)$ occurs as a weight of the $S^1$-action at some fixed point.  Then the integral cohomology ring of $M$ is $H^*(M; \Z) = \Z[x]/x^{n+1}$, where $x=[\omega]$.
\end{proposition}

\subsection{The case when the integer $\phi(P_{n-1})-\phi(P_0)=\phi(P_n)-\phi(P_1)$ is a weight but $\phi(P_n)-\phi(P_0)$ is not a weight}
 
\
\
\medskip

In this part, we use the named weight to determine the sets of weights at
$P_0$, $P_1$, $P_{n-1}$ and $P_n$, then we use that to determine  the integral cohomology ring of $M$.

 \medskip

\begin{lemma}\label{10}
Let the circle act on a compact $2n$-dimensional symplectic manifold
$(M, \omega)$ with moment map $\phi\colon M\to\R$. Assume $M^{S^1}=\{P_0, P_1, \cdots, P_n\}$ and $[\omega]$ is a {\it primitive} integral class. Then there is a weight $w = \phi(P_1)-\phi(P_0)$ from $P_0$ to $P_1$, and there is a weight $w'=\phi(P_n)-\phi(P_{n-1})$ from $P_{n-1}$ to $P_n$.
\end{lemma}
\begin{proof}
Let $-w$ be the negative weight at $P_1$. We will show that $w=\phi(P_1)-\phi(P_0)$. The connected component $C$ of $M^{\Z_w}$ containing $P_1$ has $P_1$ as an index $2$ fixed point, and must have $P_0$ as the minimum. By Lemma~\ref{JT} used on $C$, there is a weight $w$ from $P_0$ to $P_1$.
Let $S$ be the invariant gradient sphere (for any invariant metric) from $P_1$ to $P_0$. 
Since $\index (P_i) > 2$ for all $i > 1$, the restriction map
$H^2(M; \Z)\to H^2(S; \Z)$
is an isomorphism,  so $[\omega|_S]$ is primitive integral, hence 
$$\int_{S} \omega = 1.$$
Let $\gamma$ be a gradient line on $S$ from $P_0$ to $P_1$. Since $S^1$ acts on $S-\{P_0, P_1\}$ with order $w$, we have
$$\int_{S} \omega =\frac{1}{w}\int_{\gamma} i_{X_M}\omega = \frac{1}{w}\int_{\gamma} d\phi = \frac{\phi(P_1)-\phi(P_0)}{w},$$
where $X_M$ is the vector field generated by the $S^1$-action.
(The sphere $S$ may not be smooth at $P_0$, here $\int_{S} \omega$ is a pairing of homology and cohomology.) 
Hence $w =\phi(P_1)-\phi(P_0)$. We can similarly prove the other claim.
\end{proof}

We find the sets of weights at $P_0$ and $P_{n-1}$, similarly, at $P_n$ and $P_1$ as follows.

\begin{lemma}\label{0n2}
Let the circle act on a compact $2n$-dimensional symplectic manifold
$(M, \omega)$ with moment map $\phi\colon M\to\R$. Assume $[\omega]$ is an integral class and $M^{S^1}=\{P_0,  P_1, \cdots, P_n\}$. Assume 
$\phi(P_{n-1})-\phi(P_0)=\phi(P_n)-\phi(P_1)$ and this integer 
is a weight at some fixed point, and the integer $\phi(P_n)-\phi(P_0)$ is not a weight at any fixed point. Then
\begin{equation}\label{01=}
\phi(P_n)-\phi(P_{n-1}) = \phi(P_1)-\phi(P_0),
\end{equation}
$[\omega]$ is primitive integral, and for $i=0$, $1$, $n-1$, $n$, the set of weights at $P_i$ is
$$\big\{w_{ij} = \phi(P_j) -\phi(P_i)\big\}_{j\neq i, n-i}\cup\big\{w_{i, n-i}=\frac{1}{2}\big(\phi(P_{n-i})-\phi(P_i)\big)\big\}.$$
\end{lemma}

\begin{proof}
Since $\phi(P_{n-1})-\phi(P_0) =\phi(P_n)-\phi(P_1)$,  (\ref{01=}) holds.

Let $w_{0,n-1}=\phi(P_{n-1})-\phi(P_0)$.
Since $w_{0,n-1}$ only divides $\phi(P_{n-1})-\phi(P_0)$, $w_{0,n-1}$ is a weight from
$P_0$ to $P_{n-1}$. It must have multiplicity $1$, or $w_{0,n-1}$ is the strictly largest weight between $P_0$ and $P_{n-1}$. So there is a 
$\Z_{w_{0,n-1}}$-isotropy sphere $S$ with poles $P_0$ and $P_{n-1}$. Since
$$\int_S [\omega] = \frac{\phi(P_{n-1})-\phi(P_0)}{w_{0,n-1}} =1,$$ 
$[\omega|_S]$ is primitive integral. Hence $[\omega]$ is primitive integral on $M$.

Let $\{\bar w_{0i} \,|\, i\neq 0\}$ be the set of weights at $P_0$ and 
$\{\bar w_{n-1, i}\,|\, i\neq n-1\}$ the set of weights at $P_{n-1}$. Let 
$$\bar w_{0, n-1}=w_{0, n-1}\,\,\, \mbox{and}\,\,\, \bar w_{n-1, 0}=w_{n-1, 0}=-w_{0, n-1}.$$
By Lemma~\ref{10}, we may let 
\begin{equation}\label{ad}
\bar w_{01} = w_{01} = \phi(P_1)-\phi(P_0)\,\,\,\mbox{and}\,\,\, \bar w_{n-1, n} = w_{n-1, n}=\phi(P_n)-\phi(P_{n-1}),
\end{equation}
where $w_{01}$  is a weight  between $P_0$ and $P_1$, and $w_{n-1, n}$ is a weight between $P_{n-1}$ and $P_n$.

By Lemma~\ref{equalmod}, 
\begin{equation}\label{wequal}
 \{\bar w_{0i} \,|\, i\neq 0\}  =\{\bar w_{n-1, i}\,|\, i\neq n-1\}  \mod w_{0,n-1}.
\end{equation}
Note that  all the weights in these sets except $w_{0,n-1}$ and $w_{n-1, 0}$ have
absolute values smaller than $w_{0,n-1}$.
For the positive weight $\bar w_{n-1, n}= w_{n-1, n}$ at $P_{n-1}$ in (\ref{wequal}), the correspondence can only be
\begin{equation}\label{a}
\bar w_{01} = \bar w_{n-1, n} + 0 w_{0,n-1}.
\end{equation}
(which is the same as (\ref{01=}).)
For each positive weight $\bar w_{0i}$ at $P_0$, with $i \neq 0, 1, n-1$, by (\ref{wequal})
and (\ref{a}), there exists a negative weight at $P_{n-1}$, let us call it $\bar w_{n-1, i}$, with $i \neq 0, n-1, n$,  such that
\begin{equation}\label{b}
\bar w_{0i} = \bar w_{n-1, i} + w_{0,n-1}.
\end{equation}
By Lemma~\ref{nother}, each $\bar w_{0i}$ in  (\ref{b}) is a weight between $P_0$ and some $P_m$ with $m\neq 0, 1, n-1$. Each $-\bar w_{n-1, i}$ in (\ref{b}) can only be a weight between some $P_l$ and $P_{n-1}$ with $1\leq l\leq n-2$.  We will discuss (\ref{b}) in the following two possibilities.

\noindent{\bf Case $(1)$}. Assume $2\leq m\leq n-2$. Then by Lemma~\ref{mul1}, there is a $k$ with $1\leq k\leq n-2$ such that
$$\bar w_{0i} = \phi(P_k)-\phi(P_0)\,\,\,\mbox{and}\,\,\, \bar w_{n-1, i}=\phi(P_k)-\phi(P_{n-1}).$$
We rename $\bar w_{0i}$ as $w_{0k}$, and $\bar w_{n-1, i}$ as $w_{n-1, k}$, i.e., 
\begin{equation}\label{rename}
w_{0k}= \phi(P_k)-\phi(P_0)\,\,\,\mbox{and}\,\,\, w_{n-1, k}=\phi(P_k)-\phi(P_{n-1}),\,\,\,\mbox{where $1\leq k\leq n-2$}.
\end{equation}
(If $k=1$, this $w_{01}$ as a number is equal to the one in (\ref{ad}), but here it 
is supposed to be a weight between $P_0$ and some $P_m$ with $2\leq m\leq n-2$.)

\noindent{\bf Case $(2)$}. Assume $m=n$, i.e., $\bar w_{0i}$ in  (\ref{b}) is a weight between $P_0$ and $P_n$. By assumption, 
$\bar w_{0i} < \phi(P_n)-\phi(P_0)$. By Lemma~\ref{k|}, $\bar w_{0i}|\big(\phi(P_n) - \phi(P_0)\big)$, so 
$$a\, \bar w_{0i} = \phi(P_n) - \phi(P_0),\,\,\,\mbox{for some $a$ with $2\leq a\in\N$}.$$
For the corresponding $\bar w_{n-1, i}$ in (\ref{b}), we have
$$b\, \bar w_{n-1, i} = \phi(P_l)-\phi(P_{n-1}), \,\,\,\mbox{with $1\leq l\leq n-2$ and $b\in\N$}.$$

\noindent {\bf Case $(2a)$}.  Assume $b=1$.  We rename $\bar w_{n-1, i}$ as $w_{n-1, l}$, and the corresponding $\bar w_{0i}$ as $w_{0l}$, i.e., we denote
$$w_{0l} = \phi(P_l)-\phi(P_0) \,\,\,\mbox{and}\,\,\, w_{n-1, l} =\phi(P_l)-\phi(P_{n-1})\,\,\,\mbox{with $1\leq l\leq n-2$},$$
where $w_{0l}$ is a weight between $P_0$ and $P_n$, and $-w_{n-1, l}$
is a weight between $P_l$ and $P_{n-1}$.

\noindent {\bf Case $(2b)$}.  Assume $b\geq 2$. Since we also have $a\geq 2$, for (\ref{b}) to hold, the only possibility is that $a=2$, $b=2$, and $l=1$, i.e.,
$$\bar w_{0i} = \frac{1}{2}\big(\phi(P_n)-\phi(P_0)\big)\,\,\,\mbox{and}\,\,\,\bar w_{n-1, i} =\frac{1}{2}\big(\phi(P_1)-\phi(P_{n-1})\big).$$  
We rename $\bar w_{0i}$ as $\bar w_{0n}$ and $\bar w_{n-1, i}$ as  $\bar w_{n-1, 1}$, i.e., we denote
\begin{equation}\label{0n,n-11}
\bar w_{0n} = \frac{1}{2}\big(\phi(P_n)-\phi(P_0)\big)\,\,\,\mbox{and}\,\,\, \bar w_{n-1, 1} =\frac{1}{2}\big(\phi(P_1)-\phi(P_{n-1})\big). 
\end{equation}
Moreover, if it appears, each of $\bar w_{0n}$ and $\bar w_{n-1, 1}$ has multiplicity $1$. In fact, 
since $\bar w_{0n}$ can only divide $\phi(P_n)-\phi(P_0)$, if $\bar w_{0n}$ has multiplicity bigger than $1$, then
there exists an isotropy submanifold of dimension at least $4$ with only $2$ fixed points $P_0$ and $P_n$, a contradiction. 

  Claim 1: For $2\leq j\leq n-2$,  if the $w_{n-1, j}=\phi(P_j)-\phi(P_{n-1})$ in
{\bf Case $(1)$} and  {\bf Case $(2a)$} is a weight at $P_{n-1}$, then it has multiplicity $1$. 

  Claim 2: If $w_{n-1, 1}=\phi(P_1)-\phi(P_{n-1})$ occurs, it cannot have multiplicity bigger than $1$; $w_{n-1, 1}$ and $\bar w_{n-1, 1}$ cannot appear at the same time. 

  Claim 3: {\bf Case $(2b)$} must occur. 

Proof of  Claim 1:  If $|w_{n-1, j}|\geq \frac{1}{2}w_{0, n-1}$, then $w_{n-1, j}$ can only be a weight between $P_{n-1}$ and $P_j$. If taking into account (\ref{01=}), 
we see that $M^{\Z_{|w_{n-1, j}|}}$ has $P_{n-1}$ as the maximum and $P_j$ the next fixed point to $P_{n-1}$, hence  $w_{n-1, j}$ cannot have multiplicity bigger than $1$ by Lemma~\ref{ind} 
(applied to $-\phi$ on $M^{\Z_{|w_{n-1, j}|}}$).  If $|w_{n-1, j}| < \frac{1}{2}w_{0, n-1}$, then by {\bf Case $(1)$} and  {\bf Case $(2a)$}, $w_{0j} >\frac{1}{2}w_{0, n-1}$, so $w_{0j}$ can only be a  
weight between $P_0$ and $P_j$,  by Lemma~\ref{mul1}, $w_{n-1, j}$ (and $w_{0j}$) has multiplicity $1$. 

Proof of  Claim 2:  Suppose $w_{n-1, 1}$ appears with multiplicity bigger than $1$. By {\bf Case $(1)$} and  {\bf Case $(2a)$}, $w_{n-1, 1}$ is a weight between $P_{n-1}$ and $P_1$. Then $M^{\Z_{|w_{n-1, 1}|}}$ 
of dimension at least $4$ either contains only two fixed points $P_1$ and $P_{n-1}$, or it contains $4$ fixed points $P_0$, $P_1$, $P_{n-1}$ and $P_n$. The first case is not possible.  In the latter case, we have $w_{n-1, 1}|w_{01}$, hence
$w_{n-1, 1}|w_{0, n-1}$, so in $M^{\Z_{|w_{n-1, 1}|}}$, $w_{0, n-1}$ is a weight between $P_0$ and $P_{n-1}$, then $P_{n-1}$ has index at least $6$ ($P_1$ has index $2$), again not possible. Now suppose $w_{n-1, 1}$ and $\bar w_{n-1, 1}$ both appear. By  {\bf Case $(2b)$}, $\bar w_{n-1, 1}$ is (also) a weight between $P_{n-1}$ and $P_1$. 
The same argument, by considering $M^{\Z_{|\bar w_{n-1, 1}|}}$, gives a contradiction.
 
Proof of  Claim 3: Suppose {\bf Case $(2b)$} does not occur.  
Since there are $n-1$ number of negative weights at $P_{n-1}$, by Claims 1 and 2, the set of weights at $P_{n-1}$ is 
$$\big\{w_{n-1, j}=\phi(P_j)-\phi(P_{n-1})\big\}_{j\neq n-1}.$$
Similarly, by symmetry (or by using $-\phi$), the set of weights at $P_1$ is
$$\big\{w_{1j}=\phi(P_j)-\phi(P_1)\big\}_{j\neq 1}.$$
Then $\Gamma_1-\Gamma_{n-1} =(n+1)\big(\phi(P_{n-1})- \phi(P_1)\big)$. This contradicts to
Theorem~\ref{equiv} and Lemma~\ref{sub}.

By  Claims 2 and 3, $w_{n-1, 1}=\phi(P_1)-\phi(P_{n-1})$ does not appear.
We may rename $\bar w_{0n}$ as
$w_{0n}$ and $\bar w_{n-1, 1}$ as $w_{n-1, 1}$. Then the claims mean that the set of weights at $P_{n-1}$ is as claimed in the lemma, and correspondingly, the set of weights at $P_0$ is as claimed.  Similarly, by symmetry or by looking at $-\phi$, the sets of weights at $P_n$ and at $P_1$ are as claimed. 
\end{proof}

\begin{lemma}\label{ringprod2}
 Let the circle act on a compact $2n$-dimensional symplectic manifold
$(M, \omega)$ with moment map $\phi\colon M\to\R$.  Assume $[\omega]$ is a
primitive integral class and $M^{S^1}=\{P_0,  P_1, \cdots, P_n\}$. Then the following two conditions are equivalent:
\begin{enumerate}
\item $H^*(M; \Z) = \Z(x, y)/\big(x^{\frac{1}{2}(n+1)}-2y, y^2\big),$
where $n\geq 3$ is odd, $x = [\omega]$, and  $\deg(y) = n+1$.
\item $\Lambda_n^- =\frac{1}{2} \prod_{j=0}^{n-1} \big(\phi(P_j)-\phi(P_n)\big)$.
\end{enumerate}
\end{lemma}

\begin{proof}
We proceed similarly as in the proof of Proposition~\ref{0nring}.  We get that
$(2)$ holds if and only if $\frac{1}{2} [\omega]^n$ is a generator of $H^{2n}(M; \Z)$. If $(1)$ holds, then $\frac{1}{2} [\omega]^n\in H^{2n}(M; \Z)$ is a generator, so $(1)$ implies $(2)$.
Now we prove $(2)$ implies $(1)$. First, by the above we have that $\frac{1}{2} [\omega]^n\in H^{2n}(M; \Z)$ is primitive integral. By Lemma~\ref{order},  
$H^{2i}(M; \Z)=\Z$ and $H^{2i-1}(M; \Z)=0$ for all $0\leq i\leq n$, so the generator
of $H^{2i}(M; \Z)$ is a (rational) multiple of $[\omega]^i$ for $0\leq i\leq n$.
Since $[\omega]$ is also primitive integral, by Poincar\'e duality, $\frac{1}{2}[\omega]^{n-1}$ is primitive integral. 
Then we get that $\frac{1}{2} [\omega]^{n-2}$ is primitive integral, and then 
$[\omega]^2$ is primitive integral. Inductively, we get that the following classes are generators  of the corresponding integral cohomology groups:
$$1,\, [\omega], \,[\omega]^2,\, \cdots,\, [\omega]^{\frac{n-1}{2}},\, \frac{1}{2}[\omega]^{\frac{n+1}{2}},\, \cdots,\, \frac{1}{2}[\omega]^{n-1}, \,\frac{1}{2}[\omega]^n.$$
So $n\geq 3$ and is odd.
If we let $x=[\omega]$ and $y=\frac{1}{2}[\omega]^{\frac{n+1}{2}}$, we get the ring $H^*(M;\Z)$ as claimed in $(1)$.
\end{proof}

\begin{remark}
The conditions in Lemma~\ref{ringprod2} are also equivalent to:
$\Lambda_i^- = \prod_{j=0}^{i-1} \big(\phi(P_j)-\phi(P_i)\big)$  if $i \leq\frac{n-1}{2}$, and
 $\Lambda_i^- =\frac{1}{2} \prod_{j=0}^{i-1} \big(\phi(P_j)-\phi(P_i)\big)$ if $i\geq \frac{n+1}{2}$.
The proof is by considering the classes $\prod_{j=0}^{i-1} \big(\ut + \phi(P_j)t\big)$ for all fixed index $i$. Corollary~\ref{cor} gives
$$\prod_{j=0}^{i-1} \big(\ut + \phi(P_j)t\big) = a_i\Tilde\alpha_i, \,\,\,\mbox{where $a_i\in\Z$}.$$
Restricting this to ordinary cohomology and to $P_i$, we can obtain what we need.
\end{remark}

Now with Lemmas~\ref{0n2} and \ref{ringprod2}, we obtain the ring $H^*(M; \Z)$:

\begin{proposition}\label{01ring}
 Let the circle act on a compact $2n$-dimensional symplectic manifold
$(M, \omega)$ with moment map $\phi\colon M\to\R$.  Assume  $[\omega]$ is an integral class and $M^{S^1}=\{P_0,  P_1, \cdots, P_n\}$. Assume 
$\phi(P_{n-1})-\phi(P_0)=\phi(P_n)-\phi(P_1)$ and this integer is a weight at some fixed point, and $\phi(P_n)-\phi(P_0)$ is not a weight at any fixed point. Then the integral cohomology ring of $M$ is
$$H^*(M; \Z) = \Z(x, y)/\big(x^{\frac{1}{2}(n+1)}-2y, y^2\big),$$
where $n\geq 3$ is odd, $x = [\omega]$,  and  $\deg(y) = n+1$.
\end{proposition}

\section{from the known weight to all the weights and to the total Chern class $c(M)$}\label{weights}

In this section, we determine the sets of weights at all the fixed points, and
we determine the total Chern class of the manifold.

\subsection{The case when $\phi(P_n)-\phi(P_0)$ is a weight}
\

\medskip

Proposition~\ref{0ni} gives the sets of weights at all the fixed points.
As a preparation to prove Proposition~\ref{0ni}, we first prove Lemma~\ref{lempq}.

\begin{lemma}\label{lempq}
Let the circle act on a compact $2n$-dimensional symplectic manifold
$(M, \omega)$ with moment map $\phi\colon M\to\R$. Assume $[\omega]$ is an integral class and $M^{S^1} = \{P_0, P_1, \cdots, P_n\}$.  For a fixed 
$i\geq 0$, assume the set of weights at each $P_k$ with $0\leq k\leq i$
is 
$$\big\{w_{kj}=\phi(P_j)-\phi(P_k) \,|\, j\neq k\big\}.$$
If for each $k$ with $0\leq k < i$, $w_{kj}$ is a weight between $P_k$ and
$P_j$ for any $j\neq k$, then each $w_{ij}$ is a weight between $P_i$ and 
$P_j$ for any $j\neq i$.
\end{lemma}
\begin{proof}
Recall that by Lemma~\ref{k|}, if $w$ is a weight between $P_l$ to $P_m$, then $w|\big(\phi(P_m)-\phi(P_l)\big)$ for any $l$ and $m$.

First assume $i=0$. If some $w_{0j} > \frac{1}{2}w_{0n}$, then $w_{0j}$ can only be a weight from $P_0$ to $P_j$. The claim follows by Lemma~\ref{nother} and induction on the index $j$. (Or see the general argument below.) 

Now assume $i > 0$. By assumption, for each $j < i$, $w_{ij}$ is a weight
between $P_i$ and $P_j$. The assumption also implies that for each $j > i$, $w_{ij}$ is a weight from $P_i$ to some $P_l$ with $l > i$. For $j > i$, if $w_{ij} > \frac{1}{2}w_{in}$, then $w_{ij}$ only divides $\phi(P_j)-\phi(P_i)$,
so it can only be a weight from $P_i$ to $P_j$. Now assume for certain $j_0$, $w_{ij}$ is a weight from $P_i$ to $P_j$ for each $j$ with
$j_0 < j\leq n$. We claim that $w_{ij_0}$ is a weight from $P_i$ to 
$P_{j_0}$. (Note that it can only be a weight from $P_i$ to some $P_l$
with $l\geq j_0$.)
Consider the connected component $C$ of $M^{\Z_{w_{ij_0}}}$ containing $P_i$. Applying Lemma~\ref{JT} on $C$, we get that $w_{ij_0}$ is a weight from $P_i$ to some $P_l$, $l\geq j_0$, with $\index|_C (P_l) = \index|_C(P_i) + 2$. 
Suppose $l > j_0$. Then $w_{ij_0}|\big(\phi(P_l)-\phi(P_i)\big)$, while the latter is $w_{il}$, so $P_l$ has index at least $4$ in $C$.
The assumption implies that the number of fixed points in $C$ below $P_i$ determines the index of $P_i$ in $C$.
Note that $P_k\in C$ with $k < i$ if and only if $w_{ij_0} |\big(\phi(P_k)-\phi(P_i)\big)$ if and only if $w_{ij_0} |\big(\phi(P_l)-\phi(P_k)\big)$. 
So if any $P_k$ with $k<i$ is contained in $C$, then it contributes to the index of $P_i$ and $P_l$ by the same number, hence $\index|_C (P_l)\geq \index|_C (P_i) +4$, a contradiction. Hence $w_{ij_0}$ is a weight from $P_i$ to $P_{j_0}$. The claim follows by induction on $j$.
\end{proof}

The sets of weights at all the fixed points are as follows.
\begin{proposition}\label{0ni}
Let the circle act on a compact $2n$-dimensional symplectic manifold
$(M, \omega)$ with moment map $\phi\colon M\to\R$. Assume $[\omega]$ is an integral class and $M^{S^1}=\big\{P_0,  P_1, \cdots, P_n\big\}$.  Assume the integer $w_{0n}=\phi(P_n)-\phi(P_0)$ is a weight of the $S^1$-action at some fixed point. Then 
the set of weights at any $P_i$ is
\begin{equation}\label{weighti}
\big\{w_{ij} = \phi(P_j) -\phi(P_i)\big\}_{j\neq i}.
\end{equation}
That is, the sets of weights at the fixed points are isomorphic to those of the standard circle action on $\CP^n$ as in Example~\ref{CP^n}.
\end{proposition}

\begin{proof}
First of all, by Lemma~\ref{0n1}, the claim holds for $P_0$ and $P_n$; moreover, by Lemma~\ref{lempq},
$w_{0i}$ is a weight between $P_0$ and $P_i$ for each $i\neq 0$. Similarly,
$w_{ni}$ is a weight between $P_n$ and $P_i$ for each $i\neq n$.

We use induction. Assume for a fixed $i$, the claim (\ref{weighti}) holds for all $P_k$'s with $k < i$, and for each such $k$, $w_{kj} = \phi(P_j)-\phi(P_k)$ is a weight between $P_k$ and $P_j$ for any $j \neq k$.  We will prove all these hold for $k=i$. 
First, the inductive hypotheses imply that the set of negative weights at $P_i$ is
\begin{equation}\label{negi}
\big\{w_{ij} = \phi(P_j)-\phi(P_i) \big\}_{0\leq j < i}.
\end{equation}
If $P_i=P_{n-1}$, then we are done. Otherwise,
since $w_{in} = -w_{ni} > 1$  is a weight between $P_i$ and $P_n$,
by Lemma~\ref{equalmod},
\begin{equation}\label{modin}
\mbox{\{weights at $P_i\} =$\{weights at $P_n\} \mod w_{in}$}.
\end{equation}
First, in (\ref{modin}), we have the correspondence
\begin{equation}\label{0<j<i}
w_{ij} = w_{nj} + w_{in}, \,\,\,\forall \,\, j \,\,\,\mbox{with $0\leq j < i$}.
\end{equation}
Next, the inductive hypothesis and Lemma~\ref{nother} imply that each positive weight at $P_i$ other than $w_{in}$ is a weight from $P_i$ to some $P_k$ with $i < k < n$ and hence is less than $w_{in}$. For each weight $w_{nj}$ at $P_n$ with $i < j < n$, since the negative weights at $P_i$ already occured in (\ref{0<j<i}),  there can only exist a positive weight at $P_i$, denoted $w_{ij}$, such that there is the correspondence in (\ref{modin}):
$$w_{ij} = w_{nj} + w_{in}, \,\,\,\forall \,\, j \,\,\,\mbox{with $i < j < n$},$$
which yields the subset of positive weights at $P_i$:
\begin{equation}\label{posi}
\big\{w_{ij} =\phi(P_j)-\phi(P_i)\big\}_{i < j < n}.
\end{equation}
Equations (\ref{negi}) and (\ref{posi}), and $w_{in}$ give the claim  (\ref{weighti}) 
for $P_i$. By Lemma~\ref{lempq}, each $w_{ij}$ in (\ref{posi}) is a weight between $P_i$ and $P_j$. Hence
in the set of weights at $P_i$, (\ref{weighti}),  each $w_{ij}$ is a weight between $P_i$ and $P_j$.
\end{proof}

Using Proposition~\ref{0ni}, we can determine $c(M)$:

\begin{proposition}\label{weights-tc1}
Let the circle act on a compact $2n$-dimensional symplectic manifold
$(M, \omega)$ with moment map $\phi\colon M\to\R$. Assume $[\omega]$ is an integral class and $M^{S^1}=\big\{P_0,  P_1, \cdots, P_n\big\}$. Assume the integer $w_{0n}=\phi(P_n)-\phi(P_0)$ is a weight of the $S^1$-action at some fixed point. Then the total Chern class of $M$ is $c(M) = (1+[\omega])^{n+1}$, hence is isomorphic to $c(\CP^n)$.
\end{proposition}

\begin{proof}
By the injectivity theorem (\cite{K}, \cite{TW}), the restriction map
$$H^*_{S^1}(M; \Z)\to H^*_{S^1}(M^{S^1}; \Z)$$
is injective for any Hamiltonian $S^1$-manifold $M$ with isolated fixed points. 
Consider the class $\alpha = \prod_{0\leq j\leq n} \big(1+\ut+\phi(P_j)t\big)$. By Proposition~\ref{0ni}, we can check that
$$\alpha|_{P_i} = c^{S^1}(M)|_{P_i},\,\,\,\forall\,\,\, 0\leq i\leq n.$$ Hence 
$c^{S^1}(M)=\prod_{0\leq j\leq n} \big(1+\ut+\phi(P_j)t\big)$. Restricting it to ordinary cohomology, we obtain our claim. 
\end{proof}

\subsection{The case when the integer $\phi(P_{n-1})-\phi(P_0)=\phi(P_n)-\phi(P_1)$ is a weight but $\phi(P_n)-\phi(P_0)$ is not a weight}
\

\medskip
Similar to the last subsection, we first prove Lemma~\ref{ktoi} in order to  prove
Proposition~\ref{01i} on the sets of weights at all the fixed points.

\begin{lemma}\label{ktoi}
 Let the circle act on a compact $2n$-dimensional symplectic manifold
$(M, \omega)$ with moment map $\phi\colon M\to\R$, where $n\geq 3$ is odd.  
Assume $[\omega]$ is an integral class and $M^{S^1}=\{P_0,  P_1, \cdots, P_n\}$. Assume for a fixed $i$ with
$0\leq i \leq\frac{n-1}{2}$, the set of weights at each $P_k$ with $k\leq i$ and $k\geq n-i$ is
$$\big\{w_{kj} = \phi(P_j) -\phi(P_k)\big\}_{j\neq k, n-k}\cup\big\{w_{k, n-k}=\frac{1}{2}\big(\phi(P_{n-k})-\phi(P_k)\big)\big\},$$
and for each $k < i$ and $k > n-i$, $w_{kj}$ is a weight between $P_k$ and $P_j$
for each $j\neq k$. Then $w_{ij}$ is a weight between $P_i$ and $P_j$ for each 
$j\neq i$, and $w_{n-i, j}$ is a weight between $P_{n-i}$ and $P_j$ for each $j\neq n-i$.
\end{lemma}
\begin{proof}
First assume $i=0$. Clearly $w_{0n}$ is a weight between $P_0$ and $P_n$.
We claim that each $w_{0j}$ in the subset of weights $\{w_{0j}\}_{1\leq j\leq n-1}$ is a weight between $P_0$ and some $P_l$ with $1\leq l\leq n-1$. If the claim holds, then the same proof of 
Lemma~\ref{lempq} for $i=0$ yields that each $w_{0j}$ in $\{w_{0j}\}_{1\leq j\leq n-1}$ is a weight between $P_0$ and $P_j$. First, by Lemma~\ref{10}, $w_{01}$ is a weight between $P_0$ and $P_1$. Clearly, if $w_{0j} > \frac{1}{2}\big(\phi(P_n)-\phi(P_0)\big)=w_{0n}$, then $w_{0j}$ can only be a weight between $P_0$ and $P_j$. Now assume for some $j$ with $1< j \leq n-1$, $w_{0j} < \frac{1}{2}\big(\phi(P_n)-\phi(P_0)\big)$ is the first one (from bigger ones) such that it {\it can only be} a weight between $P_0$ and $P_n$, then $w_{0j} > 1$ (since $j > 1$), and $w_{0j}|\big(\phi(P_n)-\phi(P_0)\big)$. Then $w_{0j}|\big(\phi(P_n)-\phi(P_j)\big)$, so $w_{0j}|w_{nj}$. There exists $l$ so that $w_{0j} = -w_{nl}$. By the above, 
$w_{0l} > \frac{1}{2}\big(\phi(P_n)-\phi(P_0)\big)$ is a weight between
$P_0$ and $P_l$, and we have $w_{0j}|w_{0l}$.
So the connected component $C$ of $M^{\Z_{w_{0j}}}$ containing
$P_0$ and $P_n$ also contains $P_l$ and $P_j$, 
so $\index|_C (P_n)\geq \index|_C (P_0) +4$. 
Applying Lemma~\ref{JT} on $C$, we get that there is an index $2$ fixed point, say $P_m$, in $C$ such that there is a weight $w_{0j}$ between $P_0$ and $P_m$,
contradicting to that $w_{0j}$ has multiplicity $1$ at $P_0$ and that it can only be a weight between $P_0$ and $P_n$.
Hence the claim of the lemma for  $i=0$ follows. Similarly, the claim of the lemma for $n-0 = n$ follows.

For a fixed $i$ with $0 < i \leq\frac{n-1}{2}$, the assumption implies that for 
$w_{ij}$, we only need to show the claim for $i+1\leq j\leq n-i-1$.
For that, we only need to show that none of $w_{ij}$ with $i+1\leq j\leq n-i-1$ can {\it only be} a weight between $P_i$ and $P_{n-i}$. Then the claim for each $w_{ij}$ with $i+1\leq j\leq n-i-1$ will follow the same way as the proof of Lemma~\ref{lempq} for $i\neq 0$. If $w_{i, i+1}=1$, then by Lemma~\ref{JT}, there is a weight $-1$ (of multiplicity $1$) at $P_{i+1}$, so $w_{i, i+1}=1$  is a weight between $P_i$ and $P_{i+1}$.  If $w_{ij} > \frac{1}{2}\big(\phi(P_{n-i})-\phi(P_i)\big)=w_{i, n-i}$, then $w_{ij}$ can only be a weight between $P_i$ and $P_j$. 
Now assume for some first $j$ (from bigger ones) with
$i+1\leq j\leq n-i-1$, $w_{ij} > 1$ and $w_{ij}< w_{i, n-i}$ can only be a weight between $P_i$ and $P_{n-i}$.
Then  $w_{ij} > 1$ and $w_{ij}|\big(\phi(P_{n-i})-\phi(P_i)\big)$.
The rest of the proof is similar to that of the first paragrah for $i=0$, and 
for the index consideration, it is similar to the case  $i\neq 0$ in  the proof of Lemma~\ref{lempq}.
\end{proof}

The sets of weights at all the fixed points are as follows.

\begin{proposition}\label{01i}
Let the circle act on a compact $2n$-dimensional symplectic manifold
$(M, \omega)$ with moment map $\phi\colon M\to\R$. Assume $M^{S^1}=\big\{P_0,  P_1, \cdots, P_n\big\}$ and $[\omega]$ is an integral class. Assume $\phi(P_{n-1})-\phi(P_0)=\phi(P_n)-\phi(P_1)$ and this integer is a weight at some fixed point but $\phi(P_n)-\phi(P_0)$ is not a weight at any fixed point. Then the set of weights at any $P_i$ is
\begin{equation}\label{wi}
\big\{w_{ij} = \phi(P_j) -\phi(P_i)\big\}_{j\neq i, n-i}\cup\big\{w_{i, n-i}=\frac{1}{2}\big(\phi(P_{n-i})-\phi(P_i)\big)\big\}.
\end{equation}
Moreover,
\begin{equation}\label{kn=}
\phi(P_k)-\phi(P_0) = \phi(P_n) -\phi(P_{n-k}), \,\,\,\forall \,\,\,  1\leq k\leq\frac{n-1}{2}.
\end{equation}
Hence the sets of weights are isomorphic to those of the standard circle action on 
$\Gt_2(\R^{n+2})$ as in Example~\ref{grasso}. 
\end{proposition}

\begin{proof}
First, by Lemma~\ref{0n2},  the claim (\ref{wi}) holds for $P_0$, $P_{n-1}$, $P_n$ and $P_1$, and (\ref{kn=}) holds for $k=1$.

Note that by Proposition~\ref{01ring}, $n\geq 3$ is odd, so the number of fixed points $n+1\geq 4$ is even. By Lemma~\ref{ktoi}, for $i=0, 1, n-1, n$, each $w_{ij}$ 
is a weight between $P_i$ and $P_j$ with $j\neq i$.
We use induction. Fix $i$ with $2\leq i \leq \frac{n-1}{2}$. Assume that for each $k$ with $0\leq k < i$, (\ref{wi}) holds for $P_k$ and $P_{n-k}$,  $w_{kj}$ is a weight between $P_k$ and $P_j$ for each $j\neq k$, and $w_{n-k, j}$ is a weight between $P_{n-k}$ and $P_j$ for
each $j\neq n-k$, moreover, (\ref{kn=}) holds for all $1\leq k \leq i-1$.
We will prove all these hold if we replace $k$ by $i$. 
The inductive hypotheses implies that the set of negative weights at $P_i$ is 
\begin{equation}\label{negi2}
\big\{w_{ij}=\phi(P_j)-\phi(P_i)\big\}_{0\leq j\leq i-1},
\end{equation}
a subset of positive weights at $P_i$ is
\begin{equation}\label{subp}
\big\{w_{ij}=\phi(P_j)-\phi(P_i)\big\}_{n-i+1\leq j\leq n},
\end{equation}
and each of the rest of the positive weights at $P_i$ is a weight from $P_i$ to some $P_j$ with $i+1\leq j\leq n-i$, hence is less than $w_{in}=-w_{ni}$.
Since $w_{in}$  is a weight between $P_i$ and $P_n$, by Lemma~\ref{equalmod},
\begin{equation}\label{mod-in}
\mbox{\{weights at $P_i$\}=\{weights at $P_n$\}}  \mod w_{in}.
\end{equation}
Note that in (\ref{mod-in}), we first have the equalities
\begin{equation}\label{j<i}
w_{ij} = w_{nj} + w_{in}, \,\,\,\forall \,\, j \,\,\,\mbox{with $1\leq j\leq i-1$ and  $n-i+1\leq j\leq n-1$}.
\end{equation}
For the remaining negative weight $w_{i0}=-w_{0i}$ at $P_i$, by symmetry, we may assume $|w_{i0}| < w_{in}$.
All the remaining weights at $P_n$ not occuring in (\ref{j<i})  (other than $w_{ni}$) have absolute values less than $w_{in}$. 
So the remaining correspondences in (\ref{mod-in}) can only be
\begin{equation}\label{i0=}
w_{i0} = w_{n, n-i}  + 0 w_{in},\,\,\,\mbox{and}
\end{equation}
\begin{equation}\label{remain}
w_{ij} = w_{nj} + w_{in} \,\,\,\mbox{if $i+1\leq j\leq n-i-1$, and}\,\,\, w_{i,n-i}=w_{n0} + w_{in},
\end{equation}
where for the known $w_{nj}$ with $i+1\leq j\leq n-i-1$, we denote
the corresponding weight at $P_i$ by $w_{ij}$, and for the known $w_{n0}$, we denote the corresponding weight at $P_i$ by $w_{i,n-i}$.
Equation (\ref{i0=}) gives (\ref{kn=}) for $k=i$. 
From (\ref{remain}), we get the subset of positive weights at $P_i$:
\begin{equation}\label{posi2}
\big\{w_{ij} = \phi(P_j)-\phi(P_i)\big\}_{i+1\leq j\leq n-i-1}\cup \big\{w_{i, n-i}=\frac{1}{2}\big(\phi(P_{n-i})-\phi(P_i)\big)\big\}.
\end{equation}
Here,  we used (\ref{kn=}) for $k=i$ to get $w_{i, n-i}$. Equations (\ref{negi2}), (\ref{subp}), (\ref{posi2}), and $w_{in}=-w_{ni}$ give the set of weights, claim (\ref{wi}) for $P_i$.
Similarly, claim (\ref{wi}) for $P_{n-i}$ follows. Finally, by Lemma~\ref{ktoi}, each $w_{ij}$ is a weight between $P_i$ and $P_j$, and each $w_{n-i, j}$ is a weight between $P_{n-i}$ and $P_j$.
\end{proof}

Using Proposition~\ref{01i}, we obtain $c(M)$:

\begin{proposition}\label{weights-tc2}
Let the circle act on a compact $2n$-dimensional symplectic manifold
$(M, \omega)$ with moment map $\phi\colon M\to\R$. Assume $[\omega]$ is an integral class and $M^{S^1}=\big\{P_0,  P_1, \cdots, P_n\big\}$.  Assume $\phi(P_{n-1})-\phi(P_0)=\phi(P_n)-\phi(P_1)$ is a weight  at some fixed point but $\phi(P_n)-\phi(P_0)$ is not a weight for the $S^1$-action. Then the total Chern class of $M$ is 
$c(M) =\frac{(1+[\omega])^{n+2}}{1+2[\omega]}$, hence is
isomorphic to $c\big(\Gt_2(\R^{n+2})\big)$.
\end{proposition}

\begin{proof}
Similar to the proof of Proposition~\ref{weights-tc1}, the restriction map
$$H^*_{S^1}(M; \Z)\to H^*_{S^1}(M^{S^1}; \Z)$$
 is injective. 
Consider the class 
$$\alpha=\frac{\big(1+\ut +\frac{1}{2}\big(\phi(P_n)+\phi(P_0)\big)t\big)\prod_{0\leq j\leq n}\big(1+\ut +\phi(P_j)t\big)}{1+\big(\ut+\phi(P_0)t\big) +\big(\ut+\phi(P_n)t\big)}.$$
By Proposition~\ref{01i}, 
$\big(\phi(P_0) - \phi(P_i)\big) +\big(\phi(P_n) - \phi(P_i)\big)= \phi(P_{n-i})-\phi(P_i)$ for all $0\leq i\leq n$.  Using this, and  Proposition~\ref{01i}, we can check that 
$$\alpha|_{P_i}=c^{S^1}(M)|_{P_i},\,\,\,\forall \,\,\, 0\leq i\leq n.$$
Hence $c^{S^1}(M) = \alpha$.  Restricting this to ordinary cohomology, we obtain our claim.
\end{proof}

\end{document}